\documentclass[11pt,a4paper]{amsart}
\usepackage{mathrsfs}
\usepackage{amsfonts}
\usepackage[notref,notcite]{showkeys}
\usepackage{enumerate}

\usepackage{amsmath,amssymb,xspace,amsthm}
\newtheorem{theorem}{Theorem}%[section]
\newtheorem{lemma}[theorem]{Lemma}%[section]
\newtheorem{proposition}[theorem]{Proposition}%[section]
\newtheorem{corollary}[theorem]{Corollary}%[section]
\numberwithin{equation}{section}

\def\span{\operatorname{span}}

\newcommand{\C}{\ensuremath{\mathbb C}\xspace}

\renewcommand{\j}{\ensuremath{\mathfrak j}}

\newcommand{\h}{\ensuremath{\mathfrak{h}}}

\newcommand{\Z}{\ensuremath{\mathbb{Z}}\xspace}
\newcommand{\N}{\ensuremath{\mathbb{N}}\xspace}

\newcommand{\W}{\ensuremath{\mathcal{W}}\xspace}

\renewcommand{\phi}{\varphi}

\renewcommand{\leq}{\leqslant}

\def\K{\mathcal{K}}

\def\sl{\mathfrak{sl}}

\def\A{\mathcal{A}}
\def\m{g}

\begin{document}
\title[Modules over $\mathcal{W}_n$ and  $\mathfrak{sl}_{n+1}(\mathbb{C})$]{Irreducible modules over Witt algebras $\mathcal{W}_n$ and over  $\mathfrak{sl}_{n+1}(\mathbb{C})$}
\author{Haijun Tan and Kaiming Zhao}
%\date{}
\maketitle

\begin{abstract}
    In this paper, by using the ``twisting technique" we  obtain a class of new modules $A_b$ over
    the Witt algebras $\mathcal{W}_n$ from modules $A$ over the Weyl algebras $\K_n$ (of Laurent polynomials) for any $b\in\C$.
     We give the necessary and sufficient conditions for $A_b$  to be irreducible, and determine the necessary and sufficient conditions for two such irreducible $\mathcal{W}_n$-modules to  be isomorphic.
   Since $\sl_{n+1}(\C)$ is a subalgebra of $\W_n$, all the above irreducible $\W_n$-modules $A_b$ can be considered as $\sl_{n+1}(\C)$-modules. For a class of such $\sl_{n+1}(\C)$-modules, denoted by $\Omega_{1-a}(\lambda_1,\lambda_2,\cdots,\lambda_n)$ where
   $a\in\C, \lambda_1,\lambda_2,\cdots,\lambda_n \in \C^*$, we determine the necessary and sufficient conditions for these $\sl_{n+1}(\C)$-modules to be irreducible. If the  $\sl_{n+1}(\C)$-module $\Omega_{1-a}(\lambda_1,\lambda_2,\cdots,\lambda_n)$ is reducible, we prove that it has a unique nontrivial  submodule $W_{1-a}(\lambda_1, \lambda_2,...\lambda_n)$ and the quotient module is the finite dimensional $\sl_{n+1}(\C)$-module with highest weight $m\Lambda_n$ for some non-negative integer $m\in \Z_+$. The
   necessary and sufficient conditions for two  $\mathfrak{sl}_{n+1}(\C)$-modules  $\Omega_{1-a}(\lambda_1,\lambda_2,\cdots,\lambda_n)$ and  $W_{1-a}(\lambda_1, \lambda_2,...\lambda_n)$ to be isomorphic are also determined. The  irreducible $\mathfrak{sl}_{n+1}(\C)$-modules $\Omega_{1-a}(\lambda_1, \lambda_2,...\lambda_n)$ and $W_{1-a}(\lambda_1, \lambda_2,...\lambda_n)$  are new.
\end{abstract}


\vskip 5pt \noindent {\em 2000  Math. Subj. Class.:} 17B10, 17B20,
17B65, 17B66, 17B68

\vskip 10pt

\section{Introduction}

In 1992, O. Mathieu \cite{M1} classified all irreducible modules
with finite-dimensional weight spaces over the Virasoro algebra,
proving a conjecture of Kac \cite{Ka}. More precisely, Mathieu
proved that irreducible weight modules with finite-dimensional weight spaces fall into two classes:
highest/lowest weight modules and  modules of tensor fields on a
circle and their quotients.  Mazorchuk and Zhao \cite{MZ1} proved
that an irreducible weight module over the Virasoro algebra is
either a Harish-Chandra module or a module in which all weight
spaces in the weight lattice are infinite-dimensional. In \cite{CGZ,
CM, LLZ, LZ2, Zh}, some simple weight modules  over the Virasoro
algebra with infinite-weight spaces were constructed. Very recently,
the non-weight representation theory of the Virasoro algebra has
made a big progress. A lot of new non-weight  modules were obtained
in \cite{BM, LGZ, LLZ, LZ1, MW, MZ2, OW, TZ1, TZ2} by using
different methods.

As Witt algebras $\mathcal{W}_n$ with $ n>1$  are generalizations of the
Virasoro algebra, we hope to apply some techniques established for
the Virasoro algebra  to Witt algebras $\mathcal{W}_n$. One can
easily note that the algebras $W_1$ and $\mathcal{W}_n$ with $n>1$ are
dramatically different. In $2004$
 Eswara Rao \cite{E2} conjectured that irreducible modules for $\mathcal{W}_n$ with finite-dimensional weight
 spaces also fall in two classes: (1) modules of the highest weight type and  (2) modules of tensor
 fields on a torus and their quotients. Recently, Y. Billig and V. Futorny \cite{BF} proved that Rao's
 conjecture is true. There are also mixed weight modules over $\mathcal{W}_n,$ see, for example \cite{HWZ}. We have not seen any results on irreducible non-weight representations for $\mathcal{W}_n$ with $n\ge2$.
 The next natural tasks for $\mathcal{W}_n$
are studying irreducible weight representations with infinite-dimensional
spaces and irreducible  non-weight representations. In the present paper, we
will consider the latter.

There are very close relationships between representations of $\mathcal{W}_n$ and representations of
finite-dimensional simple Lie algebras. Using finite dimensional irreducible representations of $\sl_n$,
one can obtain a lot of irreducible weight representations of $\mathcal{W}_n$, see \cite{E1, Sh}.
Using irreducible weight representations of $\mathcal{W}_n$,  O. Mathieu
completed the classification of simple Harish-Chandra modules over
simple finite-dimensional Lie algebras in his remarkable paper
\cite{M2}. There are also examples of irreducible weight modules
with infinite-dimensional weight spaces, see, for example
\cite{DFO}. B. Kostant \cite{Ko} studied nonsingular Whittaker
modules for all finite-dimensional simple Lie algebras, while
McDowell \cite{Mc1,Mc2} studied singular Whittaker modules for the
finite-dimensional simple Lie algebras.

The theory of generalized Verma modules for
 finite-dimensional simple Lie algebras is another  popular
 subject, see
  \cite{KM, MS} and references therein.
R. Block \cite{Bl} classified the irreducible modules over
$\mathfrak{sl}_{2}(\C)$. And the complete classification for all
irreducible modules over other finite-dimensional simple Lie
algebras is still open.

The second purpose of the present paper is to construct a class of new
irreducible modules over
the Lie algebra $\mathfrak{sl}_{n+1}(\C)(n\ge 2)$. Before introducing the
contents of this paper we first recall some basic concepts and
notation.

\vskip 5pt We denote by $\mathbb{Z}$, $\mathbb{Z}_+$, $\mathbb{N}$
and $\mathbb{C}$ the sets of  all integers, non-negative integers,
positive integers and complex numbers, respectively. All vector
spaces and algebras in this paper are over $\C$. We denote by
$U(\mathfrak{a})$ the universal enveloping algebra of the Lie
algebra $\mathfrak{a}$ over $\C$.

\

For $n\in \N$, let $\C^n$ be the row vector space  of $1\times n$
matrices with the standard basis $\{e_1,e_2,\cdots, e_n\}.$ Let
$(\cdot|\cdot)$ be the standard symmetric bilinear form such that
$(u|v)$ is the product $uv^T\in \C$, where $u^T$ is the transpose
matrix of $v.$

Let $\mathcal{A}_n=\C[t_1^{\pm1},t_2^{\pm1},\cdots,t_n^{\pm1}]$ be
the Laurent polynomial algebra over $\C$ and $\mathcal{W}_n$ be the
Lie algebra of all derivations of $\mathcal{A}_n$,   called the
{\emph {Witt algebra of rank $n$}}. Denote
$\partial_i=t_i\frac{\partial}{\partial t_i}.$ For $r=(r_1,\cdots,
r_n)\in \Z^n,$ and $
 u=(u_1,\cdots, u_n)\in \C^n,$ let $$
t^r=t^{r_1}_1\cdots
t^{r_n}_{n},\,\,\,D(u,r)=t^r\sum_{i=1}^nu_i\partial_i.
 $$ Then $\mathcal{W}_n$ is the
linear span of the set $\{D(u,r): u\in \C^n,r\in \Z^n\}$.
The Lie
bracket in $\mathcal{W}_n$ is defined by
$$[D(u,r),D(v,s)]=D(w,r+s),$$
 where $u,v\in \C^n, r,s\in\Z^n, w=(u|s)v-(v|r)u.$
It is known that  $\mathfrak{h}=\oplus_{i=1}^n\C \partial_i$ is the
 {\emph {Cartan subalgebra}} of $\mathcal{W}_n$. A $\mathcal{W}_n$-module $V$ is called a weight
 module provided that the action  of $\mathfrak{h}$ on $V$ is
 diagonalizable.

Let
$\mathcal{K}_n$ be the simple associative algebra $\C[t_1^{\pm1},\cdots,t_n^{\pm1},\partial_1,\cdots,\partial_n]$. Then $\K_n$ is also a Lie algebra, where the Lie bracket is defined
 by $[x,y]=xy-yx, \ x,y\in \K_n.$ The corresponding Lie algebra is also denoted by $\K_n$. Clearly, $\W_n$ is a Lie subalgebra of $\K_n.$ The  associative algebra $\C[t^{\pm 1},\partial],
 \partial=t\frac{d}{dt},$ is denoted by $\mathcal{K}$. Denote by $K_i$ the  associative subalgebra
 $\C[t_i^{\pm 1},\partial_i]$ of $\mathcal{K}_n, 1\le i\le n$. We see that $\K_n=K_1\otimes K_2\otimes...\otimes K_n$ as an associative algebra. %We can naturally consider  $\K_{n-1}$ as a subalgebra of $\K_n$.

Let $n>1$ be a positive integer. In $\mathcal{W}_n$,
set \begin{equation}\begin{aligned}&e_{ij}=t_it_j^{-1}\partial_j, 1\leq i,j\leq n;\\
&e_{i,n+1}=-t_i\sum_{j=1}^n\partial_j,e_{n+1,i}=t_i^{-1}\partial_i,
1\leq i\leq n;\\
&e_{n+1,n+1}=-\sum_{j=1}^n\partial_j.\end{aligned}\end{equation} It is
well-known (for example \cite{M2}) that the  above set
$$\{e_{ij}:1\le i\ne j\le n+1\}\cup \{e_{ii}-e_{i+1,i+1}: 1\le i\le
n\}$$ is the standard basis of the Lie algebra
$\mathfrak{sl}_{n+1}(\C)$, i.e., $\mathfrak{sl}_{n+1}(\C)$ can be
considered as a subalgebra of $\mathcal{W}_n.$ From this point of
view, each $\mathcal{W}_n$-module can be seen as an
$\mathfrak{sl}_{n+1}(\C)$-module. The {\emph {Cartan subalgebra}} of
$\mathfrak{sl}_{n+1}(\C)$ here is also
$\mathfrak{h}=\span\{\partial_i\,\,|\,\,1\le i\le n\},$ i.e., $\mathfrak{sl}_{n+1}(\C)$ and $\W_n$ share the same Cartan subalgebra. Let
\begin{equation}\mathfrak{n}_{+} =\bigoplus_{1\le i<j\le n+1}\C e_{ij},
\,\,\,\, \mathfrak{n}_{-}=\bigoplus_{1\le j<i\le n+1}\C e_{ij}.\end{equation}
Then $\mathfrak{sl}_{n+1}(\C)$ has the standard triangular
decomposition $\mathfrak{sl}_{n+1}(\C)=\mathfrak{n}_{-}\oplus
\mathfrak{h}\oplus\mathfrak{n}_{+}.$

\

The present paper is organized as follows. In section 2,
using the ``twisting technique"  we obtain a class of
 modules $A_b$ over   the Witt algebras $\mathcal{W}_n$ from modules $A$ over  algebras $\K_n$ for any $b\in\C$ with the action
 $$D(u,k)\circ v=(D(u,k)+b(u|k)t^k)v,\,\,\forall\,\, u\in\C^n,k\in \Z^n, v\in A.$$
 We give the necessary and sufficient conditions for $A_b$  to be irreducible (Theorem 6), and determine the necessary and sufficient conditions for two such irreducible $\mathcal{W}_n$-modules to  be isomorphic (Theorem 11).
In section 3,
by embedding $\mathfrak{sl}_{n+1}(\C)$ into $\mathcal{W}_n$ as in (1.1) we
obtain a class of non-weight $\mathfrak{sl}_{n+1}(\C)$-modules
$\Omega_{1-a}(\lambda_1,\cdots,\lambda_n)=\C[\partial _1,\partial _2,,...,\partial _n]$  for any $\lambda_1,\cdots,\lambda_n\in\C^*, a\in\C$, where the action of $\W_n$ is given by
$$D(u, j)\circ(\prod_{i=1}^n\partial_i^{k_i})=(\sum_{i=1}^nu_i\partial_i-a(u|j))\prod_{i=1}^n\big(\lambda_i^{j_i}(\partial_i-j_i)^{k_i}\big)$$  for all $u\in\C^n, k \in\Z_+^n,   j\in\Z^n.$
 A striking property of these $\sl_{n+1}(\C)$-modules $\Omega_{1-a}(\lambda_1,\cdots,\lambda_n)$ is that they are free cyclic modules over the polynomial algebra $\C[\h]$.
We determine the
necessary and sufficient conditions for these $\sl_{n+1}(\C)$-modules to be
irreducible (Corollary 17).
If the  $\sl_{n+1}(\C)$-module $\Omega_{1-a}(\lambda_1,\cdots,\lambda_n)$ is reducible, we prove that it has a unique nontrivial submodule $W_{1-a}(\lambda_1,\cdots,\lambda_n)$ and the quotient module $\Omega_{1-a}/W_{1-a}$ is the finite dimensional $\sl_{n+1}(\C)$-module $L(m\Lambda_n)$ with highest weight $m\Lambda_n$ for some non-negative integer $m\in \Z_+$
(Theorem 16). As a by-product we see that all weight spaces of $L(m\Lambda_n)$ are 1-dimensional,  and that the $\sl_{n+1}(\C)$-modules $W_{1-a}(\lambda_1,\cdots,\lambda_n)$   are finitely generated free  modules over $\C[\h]$. The necessary and sufficient conditions for two such irreducible $\mathfrak{sl}_{n+1}(\C)$-modules  $\Omega_{1-a}(\lambda_1,\cdots,\lambda_n)$ (and $W_{1-a}(\lambda_1,\cdots,\lambda_n)$) to be isomorphic are also determined (Theorem 18 and Corollary 19).
The  irreducible $\mathfrak{sl}_{n+1}(\C)$-modules $\Omega_{1-a}(\lambda_1, \lambda_2,...\lambda_n)$ and $W_{1-a}(\lambda_1, \lambda_2,...\lambda_n)$  are new.

\section {Constructing $\mathcal{W}_n$-modules}

In this section, we will   use the ``twisting technique" to
construct $\W_n$-modules from modules over the associative algebra
$\K_n$.  This is a generalization of the results in
\cite{LZ1}. We also determine the necessary and sufficient
conditions for two such irreducible $\W_n$-modules to be isomorphic.

Let us recall the \emph{extended Witt algebra of rank n}:
$\mathfrak{W}_n=\mathcal{W}_n\ltimes \mathcal{A}_n$ from \cite{GLZ}.
The Lie bracket in $\mathfrak{W}_n$ is defined by
 $$[t^k,t^s]=0, \ [D(u,k), t^s]=(u|s)t^{k+s}, \ [D(u,k),D(v,s)]=D(w,k+s),$$
 where $u,v\in \C^n, k,s\in\Z^n, w=(u|s)v-(v|k)u.$

 Clearly, $\W_n$ is a Lie subalgebra of $\mathfrak{W}_n$ and $\mathfrak{W}_n$ is a Lie subalgebra of $\K_n$. So each module over the associative algebra $\K_n$  can be considered as a $\mathfrak{W}_n$-module.
 For  each $b\in \C$, we have the following  automorphism of $\mathfrak{W}_n$
$$\sigma_b: \mathfrak{W}_n\to \mathfrak{W}_n;$$
$$D(u,k)\mapsto D(u,k)+b(u|k)t^k, \ t^k\mapsto t^k,\,\,\forall\,\, u\in \C^n,k\in \Z^n.$$
 Note that $\sigma_b$ cannot
be extended to an isomorphism of the associative algebra $\K_n$ if
$b\ne0$.

Now we are going  to use the ``twisting technique". For any   module $A$ over the associative algebra
$\mathcal{K}_n$ we can define the $\mathcal{W}_n$-module action on
 $A$ as follows
$$D(u,k)\circ(v)=D(u,k)(v)+b(u|k)t^k(v), \ v\in  {A}, k\in \Z^n, b\in \C.$$ The resulting
$\mathcal{W}_n$-module is denoted by $ {A}_b.$ For convenience, we
denote the element $D(u,k)+b(u|k)t^k$ by $D _b(u,k)$.

\begin{lemma}Let $b\in \C, u,v\in \C^n, i, k\in\Z^n.$ In $\K_n$ we
have
\begin{equation}\begin{split}\Delta_b(u,v;i,k)&=-\frac{1}{2}D _b(u,k-i)D _b(v,i)
-\frac{1}{2}D _b(u,k+i)D _b(v,-i)\\
&\ \ \ +D _b(u,k)D _b(v,0)\\
&=b(b-1)(u|i)(v|i)t^k.\end{split}\end{equation}
\end{lemma}

\begin{proof} In $\K_n$ we have
$$\begin{aligned} D _b(u,k-i)&D _b(v,i)=
(D(u,k-i)+b(u|k-i)t^{k-i})(D(v,i)+b(v|i)t^i)\\
&=t^k((u|i)D(v,0))+D(u,0)D(v,0))+t^kb(u|k-i)D(v,0)\\
&\ \ \ +t^kb(v|i)((u|i)+(u|\partial))+b^2t^k(u|k-i)(v|i).
\end{aligned}$$
Replacing $i$ with $-i,0,$ and by simple computations, we can obtain (2.1).
\end{proof}

From this lemma, we can easily deduce the following
\begin{theorem}
Let $A$ be a simple module over the associative algebra
$\mathcal{K}_n$.  If $b\notin \{0,1\}$, then the $\mathcal{W}_n$-module ${A}_b$ is simple.
\end{theorem}
\begin{proof} It is sufficient to show that the actions of $t^k, k\in \Z^{n}$  on ${A}$
can be expressed as the actions of certain elements of
$U(\mathcal{W}_n)$ on $A_b$. Take a nonzero element $i\in \Z^n.$ Since $b\ne
0, 1,$ we can find $u,v\in \C^n$ such that
$$\theta=b(b-1)(u|i)(v|i)\ne 0.$$ Then by Lemma 1 we have
$t^k=\theta^{-1}\Delta_b(u,v;i,k)$, which means that the action of
$t^k$ on ${A}$ is just the action of the element
$$\theta^{-1}\big(-\frac{1}{2}{D}(u,k-i){D}(v,i)
-\frac{1}{2}{D}(u,k+i){D}(v,-i)+{D}(u,k){D}(v,0)\big)$$ of $U(\mathcal{W}_n)$ on $A_b$, as desired.
\end{proof}

Before dealing with the case $b=0$, let us prove  the following

\begin{lemma} Let $A$ be a simple module over the associative algebra
  $\mathcal{K}_n$ with  $A\ncong\A_n$. Then   there exists some $i, 1\le i\le n$ such that
   $\partial_i-k$ acts injectively on $A$ for all $k\in \Z$.
\end{lemma}

\begin{proof}  To the contrary, assume that for each $i,$ there exists $k_i\in \Z$ such that $(\partial_i-k_i)v_i=0$ for some nonzero $v_i\in A.$ Then $\partial_i(t_i^{-k_i}v_i)=t_i^{-k_i}((\partial_i-k_i)v_i)=0$ and $t_i^{-k_i}v_i\ne 0.$ So  $\ker(\partial_i)\ne 0$ for each $i$. Now we consider $A$ as an irreducible module over the Lie algebra $\K_n$. Since $\h$ is ad-semisimple on $\K_n$, we see that
$A$ is a weight module with respect to $\h$. There must exist a nonzero vector $v\in A$ such that $\h v=0$. Then $A=\K_nv_0=\A_nv_0$ and, obviously, $A\cong \A_n,$ contrary to the assumption that $A\ncong \A_n.$
\end{proof}

Now let us consider the case $b=0.$ We have the following

%Now we can prove the proposition.

\begin{theorem} Let $A$ be a simple module over the associative algebra
$\mathcal{K}_n$.  Then the $\mathcal{W}_n$-module $A_0$ is simple if
and only if $A$ is not isomorphic to   the natural $\mathcal{K}_n$-module $\A_n$.
\end{theorem}

\begin{proof}
Since the natural $\mathcal{W}_n$-module $\A_n$   has a proper submodule $\C,$ this implies the necessity.

Conversely, suppose $A\ncong \A_n$ as $\mathcal{K}_n$-modules. Then by Lemma 3 we know that there exists some $j$ such that %f$\partial_i-l$ acts injectively on $A$ for all $l\in \Z.$ Let $l=0$, So
for any  nonzero element $v\in A,$ $\partial_j(v)\neq 0.$ Take $y\in A$.
%be  For any two elements $y,v\in A_0$ with $v\neq 0,$ by the Claim, we can find some $\partial_j$ such that $\partial_j\circ (v)=\partial_j(v)\neq 0. $
Since $A=\mathcal{K}_n(\partial_j(v))$, there exist $f_{\alpha}\in U(\h), g_{\alpha}\in \A_n, 1\le \alpha\le \beta, $ such that
$$\begin{aligned}y=&\Big(\sum_{1\le \alpha\le \beta}f_{\alpha}g_{\alpha}\Big)(\partial_j(v))=\Big(\sum_{1\le \alpha\le \beta}f_{\alpha}(g_{\alpha}\partial_j)\Big)(v)\\
&=\Big(\sum_{1\le \alpha\le \beta}f_{\alpha}(g_{\alpha}\partial_j)\Big)\circ (v)
=\Big(\sum_{1\le \alpha\le \beta}f_{\alpha}\circ(g_{\alpha}\partial_j)\Big)\circ (v)
,\end{aligned}$$
which means that $y\in U(\mathcal{W}_n)\circ(v)$. Thus $A_0$ is an irreducible $\mathcal{W}_n$-module. This completes the proof.
\end{proof}

\

\noindent {\bf{Remark.}} If $A\cong\C[t_1^{\pm 1},\cdots, t_n^{\pm 1}]$ as $\mathcal{K}_n$-modules, then the corresponding $\mathcal{W}_n$-module $A_0$  has a proper submodule  $\C$. From \cite{Z}, we know that
the quotient module $\C[t_1^{\pm 1},\cdots, t_n^{\pm 1}]/\C$ is simple over $\mathcal{W}_n.$

\

Now we consider the case $b=1$.

\

\begin{theorem}  Let $A$ be an  irreducible module over the associative algebra $\mathcal{K}_n$. Then $\mathfrak{h}(A)$ is an irreducible   $\mathcal{W}_n$-submodule  of the $\mathcal{W}_n$-module $A_1$. Consequently, the $\mathcal{W}_n$-module $A_1$ is irreducible if and only if
 $\mathfrak{h}(A)=A$.
\end{theorem}

\begin{proof} For any $u\in\C^n$, $k\in\Z^n$ and $v\in A$ we have
\begin{equation}D(u, k)\circ v=(t^kD(u, 0)+(u|k)t^k)v=D(u, 0)(t^k v)\in  \mathfrak{h}(A).\end{equation}
So $\mathfrak{h}(A)$ is a $\mathcal{W}_n$-submodule of $A_1$. If $\mathfrak{h}(A)=0$, from (2.2) we deduce that $t^kv=0$ on $A$ for any $k\ne0$ and, consequently, $v=0$, which is impossible. So $\mathfrak{h}(A)\ne0$

%Since $A_1$ is a simple  $\mathcal{W}_n$-module, we deduce that $\mathfrak{h}(A)=A$.

 For any two elements $D(u, 0) y, v'\in\mathfrak{h}(A)$ where $u\in\C^n, y\in A$ and $ v'\neq 0,$   since $A$ is a simple  $\K_n$-module we can find    finitely many  $k\in\Z^n$ and $ f_{k}\in U(\h)$ such that
$y=\sum_{k}f_{k}t^{k}v'$. Thus, by (2.2), we have
$$\begin{aligned}D(u, 0) (y)&=D(u, 0) \Big( \sum_{k}f_{k}t^{k}v'\Big)=\sum_{k}f_{k}D(u, 0)t^kv'\\
&=\sum_{k}f_{k}\circ D(u, k)\circ v',\end{aligned}$$
which means that $D(u, 0) y\in U(\mathcal{W}_n)\circ v'$. Therefore, $\mathfrak{h}(A)$ is an irreducible    $\mathcal{W}_n$-submodule of $A_1$.
\end{proof}

\noindent{\bf{Example 1.}} If $A\cong \A_n$ as $\K_n$-modules, then $\h(A)\ne A$ and $A_1$ is not simple. But $\h(\A_n)$   is a simple $\mathcal{W}_n$-submodule of $A_1$.

\

Now we can summarize simplicity results on $\mathcal{W}_n$-module $A_b$ as follows.

\begin{theorem}\label{iso} Suppose that $b\in \C$, and $A$ is an
 irreducible module over the associative algebra  $\mathcal
{K}_n$. Then $A_b$ is irreducible as $\mathcal{W}_n$-module if and only if one
of the following holds
\begin{enumerate}[$(i).$]
\item
$b\ne 0$ or $1$;
\item  $b=1$ and $\h (A)=A$;
\item $b=0$ and $A$ is not isomorphic to the natural $\mathcal {K}_n$-module $\mathcal{A}_n$.
\end{enumerate}\end{theorem}

Next we will deal with the isomorphism problems between irreducible
$\mathcal{W}_n$-modules $A_b$ we just constructed.  In the case
$b\ne 0,1,$ we have the following

\begin{proposition}Let $b,b'\in \C$ with $ b\notin \{0,1\},$ and $A, A'$ be simple modules over the associative algebra
$\mathcal{K}_n$. Then $A_b\cong A'_{b'}$ as $\mathcal{W}_n$-modules if and
only if $b=b'$ and $A\cong A'$ as $\mathcal{K}_n$-modules.
\end{proposition}

\begin{proof} The sufficiency is obvious. We consider the necessity. Let $\psi: A_b\cong A'_{b'}$ be an isomorphism between the two $\W_n$-modules. Let $i,k\in \Z^n, u,v\in \C^n$ and
$$\begin{aligned}\Delta(u,v;i,k)&=-\frac{1}{2}{D}(u,k-i){D}(v,i)
-\frac{1}{2}{D}(u,k+i){D}(v,-i)\\
&+{D}(u,k){D}(v,0).\end{aligned} $$
For $w\in A_b,$ we have
$$\begin{aligned}b'(b'-1)(u|i)(v|i)t^k\psi(w)&=\Delta(u,v;i,k)\circ \psi(w)\\
=\psi(\Delta(u,v;i,k) \circ w)
&=b(b-1)(u|i)(v|i)\psi(t^k w).\end{aligned}$$ Taking $k=0$ and $u, v, i$ such that  $(u|i)(v|i)\ne 0,$ then $b'(b'-1)=b(b-1)\ne 0.$ So $b'\ne 0,1,$ and $\psi(t^kw)=t^k\psi(w),k\in \Z^n.$ Since $$\psi(\partial_iw)=\psi(\partial_i\circ w)=\partial_i\circ \psi(w)=\partial_i\psi(w),\ 1\le i\le n,$$ $\psi$ is a $\mathcal{K}_n$-module homomorphism which must be an isomorphism.

From
$$\begin{aligned}(D(u,k)&+b'(u|k)t^k)\psi(w)=D(u,k)\circ \psi (w)=\psi(D(u,k)\circ w)\\
&=\psi((D(u,k)+b(u|k)t^k)(w))=(D(u,k)+b(u|k)t^k)\psi(w)
\end{aligned}$$ we get $(b'-b)(u|k)t^k\psi(w)=0$ for all $ w\in A_b, u\in \C^n,k\in \Z^n,$ yielding that
 $b'=b$. So the proposition holds.
\end{proof}

In the case $b=0,$ we have the following

\begin{proposition}Let  $A, A'$ be simple modules over the associative algebra
  $\mathcal{K}_n$. Then $A_0\cong A'_{0}$ as $\mathcal{W}_n$-modules if and
  only if  $A\cong A'$ as $\mathcal{K}_n$-modules.
\end{proposition}

\begin{proof} The sufficiency is obvious. We only need to show the necessity. If $A\cong \A_n$ as $\K_n$-modules, by Theorem 4,
$A_0$ is not an irreducible $\mathcal{W}_n$-module. So $A'_0$ is not an irreducible $\mathcal{W}_n$-module. From  Theorem 4 again we see that $A'\cong \A_n$ as $\K_n$-modules, i.e., $A\cong A'$ as $\K_n$-modules.

Now we assume that  $A\ncong \A_n$ and $A'\ncong \A_n$ as $\K_n$-modules.
 From Lemma 3 we know that there exists $i$ such that $\partial_i-l$ acts injectively on $A'$ for all $l\in \Z.$ Let $v\in A_0$ be any nonzero element and let $k=(k_1,\cdots,k_n)\in \Z^n$. Let $\phi $ be a $\mathcal{W}_n$-module isomorphism from $A_0$ to $A'_0.$ We have $t^k\partial_j(\phi (v))=t^k\partial_j\circ \phi (v)=\phi (t^k\partial_j\circ v)=\phi (t^k\partial_j(v)).$ In particular,  $\phi (\partial_jv)=\partial_j\phi (v)$ for all $1\le j\le n. $ Then
    $$\begin{aligned}(\partial_i-k_i)\phi (t^kv)%=(\partial_i-k_i)\circ \phi (t^kv)%=\phi ((\partial_i-k_i)\circ t^kv)\\
      &=\phi (((\partial_i-k_i)t^k)(v))
      =\phi (t^k\partial_iv)=t^k\partial_i \phi (v)\\
      &=(\partial_i-k_i)(t^k\phi (v)),\end{aligned}$$ yielding that $\phi (t^kv)=t^k\phi (v).$ Therefore, $\phi $ is a $\mathcal{K}_n$-module isomorphism. This completes the proof.
\end{proof}

\

For $b=1,$ we have

\begin{proposition}Let  $A, A'$ be irreducible modules over the associative algebra
$\mathcal{K}_n$. Then $A_1\cong A'_{1}$ as $\mathcal{W}_n$-modules if and
only if  $A\cong A'$ as $\mathcal{K}_n$-modules.\end{proposition}

\begin{proof}
  The sufficiency is obvious. We only need to show the necessity. We consider the two cases separately:
   either $A\cong \A_n$ or $A\ncong \A_n.$

  {\bf{Case 1.}} $A\cong \A_n$ as $\K_n$-modules.

   In this case,   let $v=1\in A_1$. We see that $\partial_j v=\partial_j\circ v =0$ for all $ 1\le j\le n.$ Since $A_1\cong A'_1$,
    $A'_1$ has a nonzero element $v'\in A'$ such that
   $\partial_j v'=\partial_j\circ v'=0.$ The same arguments used in the proof of Lemma 3
   shows
   that $A'\cong \A_n\cong A$ as $\K_n$-modules.

  {\bf{Case 2.}} $A\ncong \A_n$ as $\K_n$-modules.

   From Lemma 3 we know that there exists some $i, 1\le i\le n$ such that $\partial_i-l$ acts injectively
   on $A'$ for all $l\in \Z.$ Let $\tau: A_1\rightarrow A_1'$ be a $\W_n$-module isomorphism. Clearly,
  $$\tau(\partial_jv)=\tau(\partial_j\circ v)=\partial_j\circ\tau(v)=\partial_j\tau(v), \forall\ v\in A, 1\le j\le n.$$
  Also we have
  $$\begin{aligned}\partial_i (\tau(t^kv))
  =&\tau(\partial_i\circ t^kv)=\tau((\partial_it^k)v)=\tau(t^k(\partial_i+k_i)v)
  =\tau(t^k\partial_i\circ v)\\ =&(t^k\partial_i)\circ \tau(v)=t^k(\partial_i+k_i) \tau(v)
  =\partial_i(t^k\tau(v)),\end{aligned}$$ for all $ v\in A, k=(k_1,\cdots,k_n)\in \Z^n.$ We deduce that
    $\tau(t^kv)=t^k\tau(v).$ So $\tau$ is a $\K_n$-module homomorphism between two simple modules,
   which must be an isomorphism. Thus $A\cong A',$ as desired.
 \end{proof}

\begin{proposition} Let  $n\ge 2$, and $A, A'$ be irreducible modules over the associative algebra
$\mathcal{K}_n$. Then $A_0\ncong A'_{1}$ as $\mathcal{W}_n$-modules.\end{proposition}
\begin{proof}
  To the contrary, assume that
  $\sigma: A_0\rightarrow A'_1$ is an isomorphism of the $\mathcal{W}_n$-modules. We consider the two cases separately:
   either $A'\cong \A_n$ or $A'\ncong \A_n$.

  {\bf{Case 1.}} $A'\cong \A_n$ as $\K_n$-modules.

  The same arguments used in the proof of Case 1 in Proposition 9 shows that $A\cong \A_n$ as $\K_n$-modules.
  Let $v_0=1\in A$ and $v_0'=\sigma(v_0).$ For all $k\in \Z^n$ and all $j, 1\le j\le
  n,$ we have
  $$0=\sigma(t^k\partial_j v_0)=\sigma(t^k\partial_j\circ v_0)=t^k\partial_j\circ v'_0=t^k(\partial_j+k_j)v'_0=\partial_jt^kv_0',$$
  which means that $\h(A')=0.$ This is impossible since $A'\cong \A_n$ as $\K_n$-modules.

  {\bf{Case 2.}} $A'\ncong \A_n$ as $\K_n$-modules.

  In this case, there exists some $i, 1\le i\le n$ such that $\partial_i-l$ acts injectively
  on $A'$ for all $l\in \Z.$ Since $n\ge 2,$ there exists $j, 1\le j\le n$ with $j\ne i$.
  For any $v\in A$ and $k'=(k_1,\cdots, k_n)$ with $k_i=0$, we have
    $$\partial_i\sigma(t^{k'}v)=\sigma(\partial_it^{k'}v)=\sigma(t^{k'}\partial_iv)=t^{k'}\partial_i\circ \sigma(v)=
    \partial_it^{k'}\sigma(v),$$
    which implies that $\sigma(t^{k'}v)=t^{k'}\sigma(v).$ We deduce
    that
    $$(\partial_j-k_j) \sigma(t^{k'}v)=(\partial_j-k_j)\circ \sigma(t^{k'}v)=\sigma((\partial_j-k_j)\circ t^{k'}v)
    =\sigma((t^{k'}\partial_j)\circ v)$$
    $$=t^{k'}\partial_j\circ\sigma(v)=t^{k'}(\partial_j+k_j)\sigma(v)=\partial_jt^{k'}\sigma(v)=\partial_j\sigma(t^{k'}v),
    $$which means that $k_j\sigma(t^{k'}v)=k_jt^{k'}\sigma(v)=0.$
    Thus $\sigma(v)=0$, which is impossible. Therefore,
    $A_0\ncong A'_1.$ This completes the proof.
\end{proof}

Now we can summarize isomorphism results as follows (we exclude the case of $n=1$ which was Theorem 12 in \cite{LZ1}).

\begin{theorem}\label{iso} Suppose that  $b, b'\in \C$, $n\ge 2$, and $A$ and $A'$ are
 irreducible modules over the associative algebra  $\mathcal
{K}_n$. Then $A_b\cong A'_{b'}$ as $\mathcal{W}_n$-modules if and
only if $b=b_1$ and $A\cong A'$ as $\mathcal {K}_n$-modules.
\end{theorem}

We like to point out that, unlike Theorem 12 in \cite{LZ1}, when $n\ge2$ we do not have $\W_n$-module isomorphism $(\h(A))_1\cong A_0$ for any irreducible $\K_n$-module $A$.

Before giving some concrete examples, we need first recall more results from \cite{LZ1}. All simple modules over the associative algebra $\K=\C[t^{\pm 1},\partial]$ were given in Lemmas 2 and 3 of \cite{LZ1}.
 More precisely,
 let $V$ be a simple $\K=\C[t^{\pm 1},\partial]$-module. Then  either  $$V\cong \mathcal{K}/(\mathcal{K}\cap \C(t)[\partial]\beta)$$ for an irreducible element $\beta$   in  the associative algebra $\C(t)[\partial]$ where $\C(t)$ is the fraction field of   $\C[t]$, or $V\cong \Omega(\lambda)=\C[\partial]$ for $\lambda\in \C^*$ where $\C[\partial]$ is the polynomial algebra in $\partial$, and the action of $\mathcal{K}$ on $\Omega(\lambda)$ is defined by
\begin{equation}t^j\partial^k=\lambda^j(\partial-j)^k,\ \partial\partial^k=\partial^{k+1},\forall\,\, k\in \Z_+, \ j\in \Z.\end{equation}

We remark that the classification for all irreducible modules
over  the associative algebras $\K_n$ with $n\ge 2$ is still unsolved.

Let $n\ge 2$ be an integer,   $V_1,V_2,\cdots, V_n$  be irreducible modules
over  the associative algebras $K_1, K_2, ..., K_n$
 respectively. Then  $A=V_1\otimes\cdots\otimes  V_n$ is a  simple module over the associative algebra $\mathcal{K}_n$.
Consequently,  we obtain a lot of  irreducible  $\mathcal{W}_n$-module ${A}_b$ except for very few cases (see Theorem 5).

\

\noindent{\bf{Example 2.}} Let $\alpha_i\in \C[t_i^{\pm 1}], \beta_i=\partial_i-\alpha_i, 1\le i\le n.$
 We have the irreducible $K_i$-module
$$V_i= K_i/(K_i\cap(\C(t_i)[\partial_i]\beta_i))=K_i/K_i\beta_i,$$
which has a basis $\{t_i^{k_i}: k_i\in \Z\}$ where we have identified $t_i^{k_i}$ with $t_i^{k_i}+K_i\beta_i.$
The action of $K_i$ is given by
$$\partial_i\cdot t_i^{k_i}=t_i^{k_i}(\alpha_i+k_i), t_i^{r_i}\cdot t_i^{k_i}=t_i^{r_i+k_i},\ \forall k_i,r_i\in \Z.$$
So the corresponding $\K_n$-module $A=V_1\otimes\cdots\otimes V_n$
is irreducible. Denote by ${\alpha}=(\alpha_1\cdots,\alpha_n)\in
\C[t_1^{\pm 1}]\times\cdots \times\C[t_n^{\pm 1}]$ and consider
$\C^n$ as a linear subspace of $\C[t_1^{\pm 1}]\times\cdots
\times\C[t_n^{\pm 1}]$. For any $b\in \C$ the action of  $\W_n$ on
the module $A_b$ is given by
$$D(u,r)\circ t^k=(u|{\alpha}+k+br)t^{r+k},\ \forall r,k\in \Z^n, u\in \C^n,$$
where  $(u|{\alpha})$ is defined as usual. If $b\notin \{0,1\}$, the
$\W_n$-module $A_b$ is irreducible. Clearly, $V\ncong \A_n$ if and
only if there exists
 some $i$ such that $\alpha_i\notin \Z$, so $A_0$ is irreducible if and only if $\alpha_i\notin \Z$ for some $i$. And $\h(A)=A$ if and only if there exists some $\alpha_i\in \C\backslash \Z$. Thus $A_1$ is an irreducible $\W_n$-module if and only if $\alpha_i\in \C\backslash\Z$ for some $i$. The module $A_b$ is a weight module if and only if  all $\alpha_i\in \C.$ If $A_b$ is a weight module, then each weight space is finite dimensional (see \cite{BF, Z}).

\

\noindent{\bf{Example 3.}} Let $\lambda_i\in \C^*, V_i=\Omega(\lambda_i), 1\le i\le n.$ Then $A=V_1\otimes \cdots\otimes V_n=\C[\partial _1,\partial _2,,...,\partial _n]$ is an irreducible $\K_n$-module. For any $b\in \C$ the action of $\W_n$ on $A_b$ is given by
$$D(u, j)\circ(\prod_{i=1}^n\partial_i^{k_i})=\lambda^j(\sum_{i=1}^nu_i\partial_i+(b-1)(u|j))\prod_{i=1}^n(\partial_i-j_i)^{k_i},$$
  for all $u\in\C^n, k \in\Z_+^n,   j\in\Z^n$, where $\lambda^j=\lambda_1^{j_1}\lambda_2^{j_2}...\lambda_n^{j_n}$.
Clearly, $A\ncong \A_n$ as $\K_n$-modules. Hence $A_b$ is an irreducible $\W_n$-module for all $b\in \C, b\ne 1.$ And $\h(A)\ne A$ implies that $A_1$ is reducible over $\W_n$ and $\h(A)$ is an irreducible submodule of $A_1$. For any $b\in \C$, $A_b$ is a non-weight $\W_n$-module.

\

\noindent{\bf{Example 4.}} Let $\lambda\in \C^*, \alpha\in \C\backslash \Z,$ then $V_1=\Omega(\lambda), V_2=K_2/K_2(\partial_2-\alpha)$ are irreducible modules over the associative algebras $K_1, K_2$ respectively. Take $\alpha_0,\alpha_1,\cdots, \alpha_n, a_1,\cdots,a_n\in \C^*, a_0=0$. Set $\beta=\partial_3-\sum_{i=0}^n\frac{\alpha_i}{t_3-a_i}$. Then we have the irreducible $K_3$-module
$$V_3=K_3/(K_3\cap\C(t_3)[\partial_3]\beta)=\C[t_3^{\pm 1}, (t_3-a_i)^{-1}|i=0,1,\cdots, n].$$ The action of $K_3$ on $V_3$ is given by
$$\begin{aligned}&\partial_3\cdot f(t_3)=\partial_3(f(t_3))+f(t_3)\sum_{i=0}^{n}\frac{\alpha_i}{t_3-a_i},\\
&t_3^s\cdot f(t_3)=t_3^sf(t_3),\ s\in \Z, f\in V_3\end{aligned}
$$(see \cite{GLZ2}).

 The $\K_3$-module $A=V_1\otimes V_2\otimes V_3$ satisfies $A\ncong \A_3$ and $\h(A)=A$. So for any $b\in \C$ the non-weight $\W_3$-module $A_b$ is irreducible. The action of $\W_3$ on $A_b$ is given by
 $$\begin{aligned}D(u,&r)\circ \partial_1^{k_1}\otimes t_2^{k_2}\otimes (t_3-a_i)^{k_3}\\
 &=u_1(\partial_1-r_1+br_1)\big(\lambda_1^{r_1}(\partial_1-r_1)^{k_1}\big)\otimes t_2^{r_2+k_2}\otimes t_3^{r_3}(t_3-a_i)^{k_3}\\
 &+\big(\lambda_1^{r_1}(\partial_1-r_1)^{k_1}\big)\otimes u_2(\alpha_2+k_2+br_2)t_2^{r_2+k_2}\otimes t_3^{r_3}(t_3-a_i)^{k_3}\\
  &+\big(\lambda_1^{r_1}(\partial_1-r_1)^{k_1}\big)\otimes t_2^{r_2+k_2}\otimes f(t_3) t_3^{r_3}(t_3-a_i)^{k_3},\end{aligned}
$$ where $f(t_3)=u_3\big(br_3+k_3t_3(t_3-a_i)^{-1}
 +\sum_{j=0}^n\frac{\alpha_1}{t_3-a_j}\big),$
$u=(u_1,u_2,u_3)\in \C^3, r=(r_1,r_2,r_3)\in \Z^3, k_1\in \Z_+, k_2\in \Z,$ and if $i=0$, then $k_3\in \Z,$ if $1\le i\le n,$ then $ k_i\in -\N.$

\

In the above three examples, all the $\K_n$-modules $A$ are product of $K_i$-modules. The next example is different.

\

\noindent{\bf{Example 5.}} Let $n\ge 2, \beta_i=\partial_i-t_1 t_2 \cdots t_n\in\K_n.$
 It is easy to prove that  the $\K_n$-module
$A= \K_n/(\K_n\beta_1+ ...+ \K_n\beta_n)$ is irreducible which is not a product $V_1\otimes V_2\otimes...\otimes V_n$ of any $K_i$-modules $V_i$.
We have the natural vector space isomorphism $A\cong\A_n$, but as $\K_n$-modules they are not isomorphic. It is easy to see that
$\h(A)$ has codimension $1$. So we obtain irreducible $\W_n$-modules $A_b$ for  $b\in\C$ with $b\ne1$.

\

\section{ Irreducible modules over $\mathfrak{sl}_{n+1}(\C)$}

Since Block \cite{Bl} gave a complete classification for irreducible modules over $\mathfrak{sl}_{2}(\C)$, we will assume $n\ge 2$ in this section, i.e., we will study irreducible (non-weight) modules over  $\mathfrak{sl}_{n+1}(\C)$  with $n\ge 2$ by restricting  irreducible $\mathcal{W}_n$-modules constructed in the previous section.  We will mainly study the  $\mathfrak{sl}_{n+1}(\C)$ -module structure on the
irreducible $\mathcal{W}_n$-modules $A_b$ for  $b\in\C$ and simple $\K_n$- modules $A=V_1\otimes\cdots\otimes  V_n$ where
$V_1,V_2,\cdots, V_n$  are irreducible modules over  the associative algebras $K_1, K_2, ..., K_n$ respectively.

For any $b\in\C, \lambda_i\in \C^*, 1\le i\le n$, we have the irreducible module $\Omega_i(\lambda_i)=\C[\partial_i]$ over the associative algebra $K_i$ as follows:
$$t_i^j\partial^l_i  (\partial_i^k)=\lambda_i^{j} (\partial_i-j)^{k+l}, \ j\in \Z, l,k\in \Z_+.$$
Then $\Omega(\lambda_1, \lambda_2,...\lambda_n)=\Omega_1(\lambda_1)\otimes\Omega_2(\lambda_2)\otimes...\Omega_n(\lambda_n)$ is an
 irreducible module over  the associative algebra $\K_n=K_1\otimes K_2\otimes...\otimes K_n$. We obtain the $\mathcal{W}_n$-module $\Omega_b(\lambda_1, \lambda_2,...\lambda_n)=(\Omega(\lambda_1, \lambda_2,...\lambda_n))_b$. The action of $\mathcal{W}_n$ on $\Omega_b(\lambda_1, \lambda_2,...\lambda_n)$ is as follows
  \begin{equation}D(u, j)\circ(\prod_{i=1}^n\partial_i^{k_i})=\lambda^j(\sum_{i=1}^nu_i\partial_i+(b-1)(u|j))\prod_{i=1}^n(\partial_i-j_i)^{k_i},\end{equation}
  for $u\in\C^n, k \in\Z_+^n,   j\in\Z^n$. From Example 3 we know that $\Omega_b(\lambda_1, \lambda_2,...\lambda_n)$ is an irreducible $\mathcal{W}_n$-module if and only if $b\ne1$.
  From now on we will consider the $\mathcal{W}_n$-module $\Omega_b(\lambda_1, \lambda_2,...\lambda_n)$ as an  $\mathfrak{sl}_{n+1}(\C)$-module by the embedding (1.1).  We will first prove

\begin{theorem} Let $n\ge2$ be a positive integer and let $a, \lambda_i\in \C^*, 1\le i\le n$  with $a\notin -\frac{1}{n+1}\Z_+$.  Then $\Omega_{1-a}(\lambda_1, \lambda_2,...\lambda_n)$ is  an irreducible  module over $\mathfrak{sl}_{n+1}(\C)$.
\end{theorem}

%\begin{theorem} Let $n\ge2$ be a positive integer, and let $b, \lambda_i\in \C^*, 1\le i\le n$.  Then $\Omega_{1-a}(\lambda_1, \lambda_2,...\lambda_n)$ is  an irreducible  module over $\mathfrak{sl}_{n+1}(\C)$ if and only if  $a\notin -\frac{1}{n+1}\Z_+.$
%\end{theorem}

 For convenience, we will identify $\Omega_1(\lambda_1)\otimes\cdots\otimes \Omega_n(\lambda_n)$ with the polynomial algebra $\C[\partial_1,\partial_2,\cdots,\partial_n]$ over $\C$ in  the commuting  indeterminants $\partial_1,\partial_2,\cdots,\partial_n$.
 To prove the theorem, we need the following

\begin{lemma}Let $a\in \C, \lambda_i\in \C^*, 1\le i\le n$. Then the $\mathfrak{sl}_{n+1}(\C)$-module $\Omega_{1-a}(\lambda_1, \lambda_2,...\lambda_n)$ is cyclic  with generator 1.
\end{lemma}
\begin{proof} The lemma follows from the fact that $\Omega_{1-a}(\lambda_1, \lambda_2,...\lambda_n)=U(\h)1=U(\h)\circ 1$.
\end{proof}

Now we can prove the theorem.
\begin{proof} Take a nonzero $\mathfrak{sl}_{n+1}(\C)$-submodule $W$ of $\Omega_{1-a}(\lambda_1, \lambda_2,...\lambda_n)$. By Lemma 13, it is sufficient to show that $1\in W.$
Let $f\in W$ be a nonzero element with minimal degree $p$ in $\partial_1,\cdots,\partial_{n}$.

\

{\bf Claim.} The degree $\text{deg}(f)$ of $f$ is 0.

 To the contrary, assume $p=\text{deg}(f)>0$. Then some $\partial_i,$ say $\partial_n$, has positive degree $k$. Write $f$ in the form
$$f=\sum_{j=0}^kf_j(\partial_1,\cdots,\partial_{n-1})\partial_n^j,$$
where $f_j(\partial_1,\cdots,\partial_{n-1})\in \C[\partial_1,\cdots,\partial_{n-1}]$ and $f_k(\partial_1,\cdots,\partial_{n-1})\neq 0.$ From
$$\begin{aligned}&\lambda_1(\lambda_n^{-1}e_{n,1}-e_{n+1,1})\circ (f)=\lambda_1(\lambda_n^{-1}t_nt_1^{-1}\partial_1-t_1^{-1}\partial_1)\circ (f)\\
&=\sum_{j=0}^{k}(\partial_1+a)f_j(\partial_1+1,\cdots,\partial_{n-1})\otimes ((\partial_n-1)^{j}-\partial_n^{j})\\
&=\sum_{j=1}^{k}(\partial_1+a)f_j(\partial_1+1,\cdots,\partial_{n-1})
\otimes(\sum_{i=0}^{j-1}(-1)^{j-i}\begin{pmatrix}j\\ j-i \end{pmatrix}\partial_n^{i}) \in W,
\end{aligned}$$
we have a nonzero vector  which has the same degree as $f$, and has lower degree in $\partial_n$ than $f$. Repeating the procedure a finite times, we can obtain a nonzero element of $W$ with same degree as $f$ and the degree of $\partial_n$ to be $0$. If the resulting element, still denoted by $f$,  does not belong to $\C[\partial_1],$ say $\partial_r$ appears but  all other $\partial_i, r<i\le n$ does not appear, using $\lambda_1(\lambda_r^{-1}e_{r,1}-e_{n+1,1})$ in the above computations several times, we can obtain an element has the same degree as the previous $f$ and   $\partial_r$ does not appear in $f$. After a finite number of steps, we can find a nonzero element of $W\cap \C[\partial_1]$ with degree $p, $ which is also  denoted by $f$. %Similarly, we can also make $f$ in $W\cap \C[\partial_i], 1\le i\le n.$ Consider the case $f\in W\cap \C[\partial_1].$

Assume $f=\sum_{i=0}^p \gamma_{i}\partial_1^{i} \in \C[\partial_1]$ where $ \gamma_{i}\in \C$ and $\gamma_{p}=1$.
By some computations, we have
\begin{equation}\begin{split} &\Big(-\lambda_1^{-1}e_{1, n+1}+\lambda_1e_{  n+1,1} -(e_{1,1}-e_{n+1, n+1}) \\
&\ \ \ \ -\sum_{j=2}^n\lambda_{j}\lambda_1^{-1}e_{1, j}
+\sum_{j=2}^n\lambda_{j}e_{  n+1,j}\Big)\circ f\\ &= \Big(\lambda_1^{-1}t_1\sum_{i=1}^n\partial_i+\lambda_1t_1^{-1}\partial_1 -(\partial_1+\sum_{i=1}^n\partial_i)\\
&\ \ \ \ -\sum_{j=2}^n\lambda_{j}\lambda_1^{-1}t_1t_{j}^{-1}\partial_{j}
+\sum_{j=2}^n\lambda_{j}t_{j}^{-1}\partial_{j}\Big)\circ f\end{split}\end{equation}
$$
=p(p-1 +(n+1)a)\partial_{1}^{p-1}+f_0=g\in W,$$
where $f_0\in \C[\partial_1]$ has degree less than $p-1$. Then $g\ne 0$ because of $p(p-1 +(n+1)a)\ne 0$, and $\text{deg}(g)=p-1<p$, which yields a contradiction. Thus we must have $\text{deg}(f)=0$. The claim follows.

So $1\in W$, i.e., $\Omega_{1-a}(\lambda_1, \lambda_2,...\lambda_n)$ is  an irreducible  module over $\mathfrak{sl}_{n+1}(\C)$.
\end{proof}

Next we will study  the $\mathfrak{sl}_{n+1}(\C)$-modules $\Omega_{1-a}(\lambda_1, \lambda_2,...\lambda_n)$ with  $a=-\frac{m}{n+1}$ for some non-negative integer $m$. Denote by
$$\begin{aligned}Y_0^i=1,\ Y_j^i=(\partial_i+a)&(\partial_i+a+1)\cdots(\partial_i+a+j-1),\\
& 1\le i\le n,\ j\in \N,\end{aligned}
$$
$$\begin{aligned}Y(j_1,j_2,\cdots,j_n)& =Y_{j_1}^{1}Y_{j_2}^{2}\cdots Y_{j_n}^{n}, \,\,\,\,  j_1,j_2,\cdots,j_n\in \Z_+,\\ & \text{with} \ j_1+j_2+\cdots+j_n=m+1,\end{aligned}$$
$$W_{1-a}(\lambda_1, \lambda_2,...\lambda_n)=\hskip -.5cm \sum_{\begin{matrix} j_1, j_2, \cdots, j_n\in\Z_+,\\j_1+j_2+\cdots+j_n=m+1\end{matrix}}\hskip -.5cm \C[\partial_1, \partial_2,...\partial_n]Y(j_1,j_2,\cdots,j_n).$$
For convenience, we will sometimes denote $W_{1-a}(\lambda_1, \lambda_2,...\lambda_n)$ by $W_{1-a}.$ Firstly, we have the following

\begin{lemma}  The subspace $W_{1-a}(\lambda_1, \lambda_2,...\lambda_n)$ is an $\sl_{n+1}(\C)$-submodule of $\Omega_{1-a}(\lambda_1, \lambda_2,...\lambda_n)$.
\end{lemma}

\begin{proof}
Since $\sl_{n+1}(\C)$ is generated by the elements
\begin{equation}\begin{split}&e_{i,i+1}=t_it_{i+1}^{-1}\partial_{i+1}, e_{i+1,i}=t_{i+1}t_i^{-1}\partial_i, 1\le i\le n-1;\\
&e_{n,n+1}=-t_n\sum_{j=1}^n\partial_j, \ e_{n+1,n}=t_n^{-1}\partial_n,\end{split}\end{equation}
and  $W_{1-a}$ has a basis
\begin{equation}\partial^k
Y(j_1,j_2,\cdots,j_n), \,\,  k \in \Z_{+}^n,\end{equation}
for all $j_1,j_2,\cdots,j_n\in\Z_+$ with $j_1+j_2+\cdots+j_n=m+1$, we just need to show that the elements  in (3.3) map the vectors in (3.4) into $W_{1-a}$ itself.

Take a vector $\partial^k Y(j_1,j_2,\cdots,j_n)$ in (3.3). Let $1\le r\ne l\le n+1$ and
$$L_{r,l}=e_{r,l}\circ
\partial^k Y(j_1,j_2,\cdots,j_n).$$ If  $r,l\le n$ and  $j_r=j_l=0$, then
$$\begin{aligned}L_{r,l}&=\lambda_r\lambda_l^{-1}\left(\partial_1^{k_1}\cdots(\partial_r-1)^{k_r}\cdots(\partial_l+a)(\partial_l+1)^{k_l}
\cdots\partial_n^{k_n}\right)\\
&\ \ \ \ \times Y(j_1,j_2,\cdots,j_n)\in W_{1-a}.\end{aligned}$$
If  $r<l\le n,$ and $j_r>0, j_l=0$, then
$$\begin{aligned}
L_{r,l}&=\lambda_r\lambda_l^{-1}\left(\partial_1^{k_1}\cdots(\partial_r-1)^{k_r}\cdots(\partial_l+a)(\partial_l+1)^{k_l}
\cdots\partial_n^{k_n}\right)\\
 &\ \ \ \ \ \times \left(Y_{j_1}^{1}\cdots Y_{j_{r}-1}^{r}(\partial_{r}+a-1)\cdots Y_{j_n}^{n}\right)\\
&= \left(\partial_1^{k_1}\cdots(\partial_{r}+a-1)(\partial_r-1)^{k_r}\cdots(\partial_l+1)^{k_l}\cdots\partial_n^{k_n}\right)\\
&\ \ \ \ \ \times\left(Y_1^lY_{j_1}^{1}\cdots Y_{j_r-1}^{r}\cdots Y_{n}^{n}\right)\in W_{1-a}.\end{aligned}$$
Similarly, if  $r<l\le n,$ and $j_r=0, j_l>0$,   we can also have $L_{r,l}\in W_{1-a}$.
 If  $r<l\le n,$ and $j_r>0, j_l>0$,  then
$$\begin{aligned}%(\lambda_r^{-1}\lambda_lt_rt_l^{-1}\partial_l)
%\left(\partial_1^{k_1}\partial_2^{k_2}\cdots\partial_n^{k_n}
%Y_{j_1}^{i_1}Y_{j_2}^{i_2}\cdots Y_{j_s}^{i_s}\right)\\
L_{r,l}&=\lambda_r\lambda_l^{-1}\left(\partial_1^{k_1}\cdots(\partial_r-1)^{k_r}\cdots(\partial_l+a)(\partial_l+1)^{k_l}
\cdots\partial_n^{k_n}\right)  \\
&\times \left(Y_{j_1}^{1}\cdots Y_{j_{r}-1}^{r}(\partial_{r}+a-1)\cdots (\partial_{l}+a+1)\cdots(\partial_{l}+a+j_l)\cdots Y_{n}^{n}\right)\\
&=\left(\partial_1^{k_1}\cdots(\partial_{r}+a-1)(\partial_r-1)^{k_r}\cdots(\partial_l+1)^{k_l}
\cdots \partial_n^{k_n}\right)\\
& \times
\left(Y_{j_1}^{1}\cdots Y_{j_r-1}^{r}\cdots Y_{j_l+1}^{l}\cdots Y_{j_n}^{n}\right)\in W_{1-a}
.\end{aligned}$$
If $j_n=0$, then
$$L_{n+1,n}=%\circ\left(\partial_1^{k_1}\partial_2^{k_2}\cdots\partial_n^{k_n}
%Y_{j_1}^{i_1}Y_{j_2}^{i_2}\cdots Y_{j_s}^{i_s}\right)\\
\lambda_n^{-1}\left(\partial_1^{k_1}\cdots\partial_{n-1}^{k_{n-1}}
(\partial_n+a)(\partial_n+1)^{k_n}\right)Y(j_1,j_2,\cdots,j_n)\in W_{1-a},$$
$$
L_{n,n+1}=-\lambda_n\Big((\sum_{j=1}^n\partial_j-a)\partial_1^{k_1}\partial_2^{k_2}\cdots(\partial_n-1)^{k_n}\Big)
Y(j_1,j_2,\cdots,j_n)\in W_{1-a}.
$$
If $j_n>0$, using  $a=-\frac{m}{n+1}$ and $ j_1+j_2+\cdots+j_n=m+1$, we obtain that
$$\begin{aligned}L_{n+1,n}&=\lambda_n^{-1}\left(\partial_1^{k_1}\cdots\partial_{n-1}^{k_{n-1}}
(\partial_n+a)(\partial_n+1)^{k_n}\right)\\
&\ \ \ \ \ \times \left(Y_{j_1}^{1}\cdots Y_{j_{n-1}}^{n-1} (\partial_{n}+a+1)\cdots(\partial_{n}+a+j_n)\right)\\
&=\lambda_n^{-1}\left(\partial_1^{k_1}\cdots\partial_{n-1}^{k_{n-1}}
(\partial_n+a+j_n)(\partial_n+1)^{k_n}\right)\\
&\ \ \ \ \ \times Y(j_1,j_2,\cdots,j_n)\in W_{1-a},%\left(Y_{j_1}^{ 1}Y_{j_2}^{ 2}\cdots Y_{j_n}^{n}\right)\in W_a,
\end{aligned}$$
$$\begin{aligned}
&L_{n,n+1}=-\lambda_n\Big(\sum_{l=1}^n\partial_l-a\Big)\\
&\ \ \times \left(\partial_1^{k_1}\cdots\partial_{n-1}^{k_{n-1}}(\partial_n+a-1)(\partial_n-1)^{k_n}\right)\left(Y_{j_1}^{ 1}\cdots Y_{j_{n-1}}^{{n-1}}Y_{j_{n}-1}^{n}\right)\\
&=-\lambda_n \big(\sum_{l=1}^{n-1}(\partial_l+a+j_l)+(\partial_n+a+j_n-1)\big)\big(Y_{j_1}^{ 1}\cdots Y_{j_{n-1}}^{{n-1}}Y_{j_{n}-1}^{n}\big) X\\
%&\ \ -\lambda_n\Big(\sum_{u=1}^{s-1}(\partial_{i_u}+a+j_u)(Y_{j_1}^{ 1}\cdots Y_{j_{n-1}}^{i_{n-1}}Y_{j_{n}-1}^{n})\Big)X\\
%&\ \ -\lambda_n\Big(\Big(\partial_n-a-(n-s)a-\sum_{u=1}^{s-1}(s+j_u)\Big)(Y_{j_1}^ 1}\cdots Y_{j_{n-1}}^{i_{n-1}}Y_{j_{n}-1}^{n})\Big)X\\
&=-\lambda_nX\sum_{l=1}^{n-1}Y(j_1,...,{j_{l}+1},..., j_n-1) -\lambda_nXY(j_1,...,  j_n)  \in W_{1-a},
\end{aligned}$$ where $X=
\partial_1^{k_1}\cdots\partial_{n-1}^{k_{n-1}}(\partial_n+a-1)(\partial_n-1)^{k_n}.$
Thus  $W_{1-a}$ is a submodule of $\Omega_{1-a}(\lambda_1, \lambda_2,...\lambda_n)$.
\end{proof}

\begin{lemma} The $\sl_{n+1}(\C)$-module $W_{1-a}(\lambda_1, \lambda_2,...\lambda_n)$ can be generated by $Y(j_1,j_2,\cdots,j_n) $ for any $j_1,j_2,\cdots,j_n\in \Z_+, \text{with} \ j_1+j_2+\cdots+j_n=m+1.$\end{lemma}

\begin{proof} The lemma follows from the following computations for all different $1\le r, s\le n$, and any $j_r, j_s\in\Z_+$:
$$\begin{aligned}\frac{1}{j_r+ 1}&(\lambda_se_{n+1, s}-\lambda_1^{-1}\lambda_se_{r,s})\circ(Y_{j_r+1}^rY_{j_s}^s)\\
=&\frac{1}{j_r+1}(\lambda_st_s^{-1}\partial_s-\lambda_r^{-1}\lambda_st_rt_s^{-1}\partial_s)\circ(Y_{j_r+1}^rY_{j_s}^s)\\
=&\frac{1}{j_r+ 1}(Y^r_{j_r+1 }Y^s_{{j_s}+1}-(\partial_r+a-1)Y_{j_r}^rY_{{j_s}+1}^s)\\
=&\frac{1}{j_r+ 1}((\partial_r+a+j_r)-(\partial_r+a-1))Y^r_{j_r}Y^s_{{j_s}+1}\\
=&Y_{j_r}^rY_{{j_s}+1}^s , 1< s\le n, 0\le j_r\le m.\end{aligned}$$

\end{proof}

Under the standard basis of $\sl_{n+1}({\C})$ in (1.1), we define the fundamental weights $\Lambda_i\in \h^*$ as follows:   $\Lambda_i(e_{jj}-e_{j+1,j+1})=\delta_{i,j}$ for all $ 1\le i, j\le n$.

\begin{theorem} Let $n\ge 2$, $m\in\Z_+$, $a=-\frac{m}{n+1}$, and $\lambda_i\in \C^*, 1\le i\le n.$ Then the $\sl_{n+1}({\C})$-module $\Omega_{1-a}(\lambda_1, \lambda_2,...\lambda_n)$ has a unique proper (irreducible) submodule $W_{1-a}(\lambda_1, \lambda_2,...\lambda_n)$, and the  quotient module $$\Omega_{1-a}(\lambda_1, \lambda_2,...\lambda_n)/W_{1-a}(\lambda_1, \lambda_2,...\lambda_n)$$ is an   irreducible module over $\sl_{n+1}(\C)$ of  dimension $\binom{m+n}m$. More precisely, the quotient module is isomorphic to the irreducible highest module  with highest weight $m\Lambda_n$.
\end{theorem}
\begin{proof} Clearly, the dimension of $\Omega_{1-a}(\lambda_1, \lambda_2,...\lambda_n)/W_{1-a}$ is the number of nonnegative solutions of the following equation $x_1+x_2+...+x_n\le m$, which is, by  induction on $m$,
$$\begin{pmatrix}n-1\\0\end{pmatrix}+\begin{pmatrix}n\\1\end{pmatrix}+\begin{pmatrix}n+1\\2\end{pmatrix}+\cdots+\begin{pmatrix}n+m-1\\m\end{pmatrix}=
\begin{pmatrix}n+m\\m\end{pmatrix}.$$

 Let us first prove that $\Omega_{1-a}(\lambda_1, \lambda_2,...\lambda_n)/W_{1-a}$ is irreducible. Take an arbitrary nonzero element $f$ in the quotient. We may assume that $f$ is a polynomial in $\partial _1, ..., \partial _n$ with degree less that $m+1$. Using the same arguments in the proof of Theorem 12, we obtain that $1\in \Omega_{1-a}(\lambda_1, \lambda_2,...\lambda_n)/W_{1-a}$. Thus
$\Omega_{1-a}(\lambda_1, \lambda_2,...\lambda_n)/W_{1-a}$ is an irreducible   module over $\sl_{n+1}(\C).$

By simple computations in $\Omega_{1-a}(\lambda_1,\lambda_2,\cdots,\lambda_n)/W_{1-a}$ we can obtain
$$\begin{aligned}(e_{11}&-e_{22})\circ Y_m^1=(\partial_1-\partial_2)Y_m^1\\
&=\big((\partial_1+a+m)-(\partial_2+a)-m\big)Y_m^1\\
&=Y_{m+1}^1-Y_m^1Y_1^2-mY_m^1=-mY_m^1,\\
(e_{ii}&-e_{i+1,i+1})\circ Y_m^1=(\partial_i-\partial_{i+1})Y_m^1\\
&=\big((\partial_i+a)-(\partial_{i+1}+a)\big)Y_m^1\\
&=Y_m^1Y_1^i-Y_m^1Y_1^{i+1}=0,\ 2\le i\le n-1,\\
(e_{nn}&-e_{n+1,n+1})\circ Y_m^1=(\partial_n+\sum_{j=1}^n\partial_j)Y_m^1\\
&=\big((\partial_1+a+m)+2(\partial_n+a)+\sum_{i=2}^{n-1}(\partial_i+a)\big)Y_m^1\\
%&\ \ \ -(m+(n+1)a)Y_m^1\\
&=Y_{m+1}^1+2Y_m^1Y_1^n+\sum_{i=2}^{n-1}Y_m^1Y_1^i=0,
\end{aligned}$$
and
$$\begin{aligned}&e_{21}\circ Y_m^1=t_2t_1^{-1}\partial_1\circ Y_m^1=\lambda_1^{-1}\lambda_2Y_{m+1}^1=0,\\
&e_{i+1,i}\circ Y_m^1=t_{i+1}t_i^{-1}\partial_i\circ Y_m^1=\lambda_i^{-1}\lambda_{i+1}Y_m^1Y_1^i=0,\ 2\le i\le n-1,\\
&e_{n+1,n}\circ Y_m^1=t_n^{-1}\partial_n\circ Y_m^1=\lambda_n^{-1}Y_m^1Y_1^n=0,
\end{aligned}$$ where we have identified the elements in $\Omega_{1-a}(\lambda_1,\lambda_2,\cdots,\lambda_n)$ with their images in  the quotient
$\Omega_{1-a}(\lambda_1,\lambda_2,\cdots,\lambda_n)/W_{1-a}$. So we can see that $Y_m^1$ is the lowest weight vector of the quotient module with weight $-m\Lambda_1,$ and hence the quotient module is isomorphic to the irreducible highest weight module  with highest weight $m\Lambda_n.$

Now let us prove that $W_{1-a}$ is irreducible. Let $f\in W_{1-a}$ be a nonzero element with minimal degree. Then by the same arguments used in the proof of Theorem 12 again and   using the fact that
$(n+1)a+(m+1)-1=0,$ we can assume that $f\in W_{1-a}\cap \C[\partial_1]$ with degree $m+1$. So $f=\gamma Y_{m+1}^1$ for some $\gamma\in \C^*.$ By Lemma 15 we know that $f$ generates $W_{1-a}$ as an  $\sl_{n+1}(\C)$-module, i.e., $W_{1-a}$ is irreducible.

Next we prove that $W_{1-a}$ is the only nontrivial $\sl_{n+1}(\C)$-submodule of $\Omega_{1-a}(\lambda_1, \lambda_2,...\lambda_n)$.
Let $W$ be a nonzero submodule of $\Omega_{1-a}(\lambda_1, \lambda_2,...\lambda_n)$. Then  $ W\cap W_{1-a}=0$ or $ W_{1-a}\subseteq W$.

If $ W_{1-a}\subseteq W$, using the irreducibility of the modules we deduce that $W= W_{1-a}$ or $W=\Omega_{1-a}(\lambda_1, \lambda_2,...\lambda_n)$.

If $ W\cap W_{1-a}=0$, then $\dim W< \infty$. We know that  $\Omega_{1-a}(\lambda_1, \lambda_2,...\lambda_n)$ is a free
$\h$-module which yields a contradiction to $\dim W< \infty$. So this case does not occur. This completes the proof.
\end{proof}

Combining Theorems 12 and 16, we have the following

\begin{corollary} Let $n\ge 2$, $a\in \C, \lambda_i\in \C^*, 1\le i\le n.$ Then the $\sl_{n+1}(\C)$-module $\Omega_{1-a}(\lambda_1, \lambda_2,...\lambda_n)$ is irreducible if and only if  $a\notin -\frac{1}{n+1}\Z_+.$
\end{corollary}

Now let us consider the problem of isomorphisms between two irreducible $\sl_{n+1}(\C)$-modules we just obtained.
Firstly, we have the following

\begin{theorem}Let $a,a',\lambda_i,\lambda_i'\in \C^*, 1\le i\le n.$ Then
$$\Omega_{1-a}(\lambda_1, \lambda_2,...\lambda_n)\cong\Omega_{1-a'}(\lambda_1', \lambda_2',...\lambda_n')$$%=\Omega_1(\lambda_1)\otimes\cdots\otimes \Omega_n(\lambda_n)\\&\cong \Omega_1(\lambda'_1)\otimes\cdots\otimes \Omega_n(\lambda'_n)=$$
if and only if $a=a', \lambda_i=\lambda_i', 1\le i\le n.$
\end{theorem}

\begin{proof} The sufficiency is obvious. We need only to consider the necessity.
Let $\pi:\Omega_{1-a}(\lambda_1,...,\lambda_n)\rightarrow \Omega_{1-a'}(\lambda_1', \lambda_2',...\lambda_n')$ be an $\sl_{n+1}(\C)$-module isomorphism.
Since $\Omega_{1-a}(\lambda_1,...,\lambda_n)=\C[\h]\circ 1$, $\Omega_{1-a'}(\lambda'_1,...,\lambda'_n)=\C[\h]\circ \pi(1)$.
So $\pi(1)\in\C^*$. Denote $\pi(1)$ by $\gamma$.
From
$$\begin{aligned}0&=\pi((\lambda_i^{-1}e_{ij}-e_{n+1,j})\circ(1))
=(\lambda_i^{-1}e_{ij}-e_{n+1,j})\circ(\pi(1))\\
&=\lambda_j'^{-1}(\partial_j+a')(\frac{\lambda_i'}{\lambda_i}-1)\gamma ,\end{aligned}
$$ we deduce that $\lambda_i=\lambda_i', 1\le i\le n.$ From
$$\begin{aligned}\gamma \partial_i&=\partial_i\circ \gamma=\pi(\partial_i\circ 1)=\pi(\partial_i)=\pi(\lambda_ie_{n+1,i}\circ (1)-a)\\
&=\lambda_ie_{n+1,i}(\gamma)-a\gamma=\gamma(\partial_i+a'-a),\ 1\le i\le n,
\end{aligned}$$ we obtain $a=a'.$ This completes the proof.
\end{proof}

Note that in the theorem we do not need that $\Omega_{1-a}(\lambda_1, \lambda_2,...\lambda_n)$ and $\Omega_{1-a'}(\lambda_1', \lambda_2',...\lambda_n')$ are irreducible.

\

In the case that $\Omega_{1-a}(\lambda_1, \lambda_2,...\lambda_n)$ is reducible, we also have the following

\begin{corollary} Let $\lambda_i,\lambda_i'\in \C^*, a,a'\in -\frac{1}{n+1}\Z_+$. Then as $\mathfrak{sl}_{n+1}(\C)$-modules,
  $W_{1-a}(\lambda_1, \lambda_2,...\lambda_n)\cong W_{1-a'}(\lambda'_1, \lambda'_2,...\lambda'_n)$ if and only if $a=a', \lambda_i=\lambda_i', 1\le i\le n.$
\end{corollary}

\begin{proof} The sufficiency is clear. Now we consider the necessity.
  Let $a=-\frac{m}{n+1}, a'=-\frac{m'}{n+1}$ for $m\ge m'\in\Z_+$. Assume that $\varrho: W_{1-a}\rightarrow W_{1-a'}$ is an $\mathfrak{sl}_{n+1}(\C)$-module isomorphism.  Denote $\varrho(Y_{m+1}^1)=\omega\in W_{1-a'}$.

\

{\bf Claim.} $\omega\in \C[\partial_1]$.

To the contrary, assume $\omega\notin \C[\partial_1].$ Then some other $\partial_i$ has positive degree in $\omega$. We may assume $\partial_n$ has degree $\beta>0$. Write $\omega$ in the form $\omega=\sum_{\alpha=0}^{\beta}\psi_{\alpha}(\partial_1,\cdots,\partial_{n-1})\partial_n^{\alpha}$, where
$\psi_{\alpha}\in \C[\partial_1,\cdots,\partial_{n-1}]$ and $\psi_{\beta}\ne 0.$ Then
$$\begin{aligned}0&=\varrho((\lambda_n^{-1}t_nt_1^{-1}\partial_1-t_1^{-1}\partial_1)(Y_{m+1}^1))
=(\lambda_n^{-1}t_nt_1^{-1}\partial_1-t_1^{-1}\partial_1)\varrho(Y_{m+1}^1)\\
&=\lambda_1'^{-1}(\partial_1+a')\sum_{\alpha=0}^{\beta}\psi_{\alpha}
(\partial_1+1,\partial_2,\cdots,\partial_{n-1})(\frac{\lambda_n'}{\lambda_n}
(\partial_n-1)^{\alpha}-\partial_n^{\alpha})\neq 0,
\end{aligned}$$ which is absurd. Thus $\omega\in \C[\partial_1]$, as desired.

Since  $$\begin{aligned}0&=\varrho((\lambda_i^{-1}e_{i1}-e_{n+1,1})\circ(Y_{m+1}^1))
=(\lambda_i^{-1}e_{i1}-e_{n+1,1})\circ \varrho(Y_{m+1}^1)\\
&=(\frac{\lambda'_i}{\lambda_i}-1)e_{n+1,1}\circ(\varrho(Y_{m+1}^1)),\ 1<i\le n,\end{aligned}$$ and $e_{n+1,1}\circ(\varrho(Y_{m+1}^1))\ne 0,$ we deduce that $\lambda_i=\lambda_i', 1< i\le n. $

The same arguments used above can deduce that $\varrho(Y_{m+1}^j)\in \C[\partial_j]$ and $\lambda_j=\lambda_j', 1\le j\le n.$

In (3.2) replacing $W$ with $W_{1-a}$, $f$ with   $Y^1_{m+1}$, the result is $0$ since elements in $W_{1-a}$ has at least degree $m+1$. So, replacing $f$ in (3.2) with  $\omega$, we should get zero, which implies that  $\omega=\gamma Y'^1_{m'+1}$, where $\gamma\in \C^*$ and
$Y'^1_{m'+1}=(\partial_1+a')(\partial_1+a'+1)\cdots(\partial_1+a'+m').$ Since
$$\begin{aligned}\varrho(Y_{m-j}^1&Y^i_{j+1})
=\varrho(\frac{1}{m+1-j}(\lambda_ie_{n+1,i}
-\lambda_1^{-1}\lambda_ie_{1i})\circ(Y_{m+1-j}^1Y_j^i))\\
&=\frac{1}{m+1-j}(\lambda_ie_{n+1,i}
-\lambda_1^{-1}\lambda_ie_{1i})\circ\varrho(Y_{m+1-j}^1Y_j^i), 1< i\le n,
\end{aligned}$$  by induction on $j$ we can deduce that $\varrho(Y^1_{m+1-j}Y^i_{j})=\gamma_jY'^1_{m'+1-j}Y'^i_{j}, 0\le j\le m'+1,$ where $\gamma_j\in \C^*.$ If $m>m',$ then
$$\begin{aligned}0&\ne \varrho(\frac{1}{m+1-(m'+1)}(\lambda_ie_{n+1,i}
-\lambda_1^{-1}\lambda_ie_{1i})\circ(Y_{m+1-(m'+1)}^1Y_{m'+1}^i))\\
&=\frac{\gamma_{m'+1}}{m-m'}(\lambda_ie_{n+1,i}
-\lambda_1^{-1}\lambda_ie_{1i})\circ(Y'^1_0Y'^i_{m'+1})=0,\end{aligned}$$ which is impossible. Thus $m=m'$ and $a=a'.$ This proves the theorem.
\end{proof}

At last, we like to compare our irreducible $\mathfrak{sl}_{n+1}(\C)$-modules just constructed in this section
%$\Omega_{1-a}(\lambda_1, \lambda_2,...\lambda_n)$ and $W_{1-a}(\lambda_1, \lambda_2,...\lambda_n)$
with other known non-weight $\mathfrak{sl}_{n+1}(\C)$-modules.
With respect to a fixed triangular decomposition of $\mathfrak{sl}_{n+1}(\C)$ as in (1.2), denote by
$\mathfrak{b}=\mathfrak{h}\oplus \mathfrak{n}_{+}$ the fixed {\emph
{Borel subalgebra}} of $\mathfrak{sl}_{n+1}(\C)$. Let
$\mathfrak{p}\supset \mathfrak{b}$ be a {\emph {parabolic
subalgebra}} of $\mathfrak{sl}_{n+1}(\C)$. Denote by $\mathfrak{n}'$
the nilpotent radical of $\mathfrak{p}$ and by $\mathfrak{u}$ the
{\emph {Levi factor}} of $\mathfrak{p}$. Then $ \mathfrak{n}'\subset
\mathfrak{n}_{+}$ and $\mathfrak{p}=\mathfrak{u}\oplus
\mathfrak{n'}$. Let $V$ be a simple $\mathfrak{p}$-module,
annihilated by $\mathfrak{n}'$. The induced module
$$M_{\mathfrak{p}}(V)=U(\mathfrak{g})\otimes_{U(\mathfrak{p})}V$$
is called the {\emph {generalized Verma module}}   of $\mathfrak{sl}_{n+1}(\C)$ associated with $\mathfrak{p}$
and $V$, see \cite{KM}. If $\mathfrak{p}=\mathfrak{b},$ then
$\mathfrak{u}=\mathfrak{h}$, $\mathfrak{n}'=\mathfrak{n}_+$, and
$M_{\mathfrak{p}}(V)$ is a usual {\emph {Verma module}} over
$\mathfrak{sl}_{n+1}(\C)$.

 In \cite{Ko, Mc1,Mc2},   irreducible Whittaker modules over
 finite-dimensional simple Lie algebras were determined.

On  our irreducible modules $\Omega_{1-a}(\lambda_1, \lambda_2,...\lambda_n)$ and $W_{1-a}(\lambda_1, \lambda_2,...\lambda_n)$ over $\mathfrak{sl}_{n+1}(\C)$, the action of any $e_{i,j}$ with $i\ne j$ is not locally finite, and these modules are finitely generated free $\C[\h]$-modules. Thus they are not
generalized Verma module or Whittaker modules over $\mathfrak{sl}_{n+1}(\C)$.
Therefore, we can conclude that  irreducible $\mathfrak{sl}_{n+1}(\C)$-modules $\Omega_{1-a}(\lambda_1, \lambda_2,...\lambda_n)$ and $W_{1-a}(\lambda_1, \lambda_2,...\lambda_n)$  are new.

\

\noindent {\bf Acknowledgement.} The second author is partially
supported by NSF of China (Grant 11271109), NSERC, and University Research Professor grant at Wilfrid Laurier University.
The authors thank Prof. R. Lu for a lot of helpful discussions during the
preparation of the paper. We thank Prof. V. Mazorchuk to send us the paper \cite{N} right after we have finished the present paper, where in \cite{N} a complete classification of all irreducible $\mathfrak{sl}_{n+1}(\C)$-modules on $\C[\h]$ has been obtained.

\vspace{10mm}

\noindent H. Tan: Department of Applied Mathematics, Changchun University of Science and Technology, Changchun, Jilin,
130022, P.R. China.
and College of Mathematics and Information Science,
Hebei Normal (Teachers) University, Shijiazhuang, Hebei, 050016 P.
R. China. Email: hjtan9999@yahoo.com

\vspace{0.2cm}
 \noindent K. Zhao: Department of Mathematics, Wilfrid
Laurier University, Waterloo, ON, Canada N2L 3C5, and College of
Mathematics and Information Science, Hebei Normal (Teachers)
University, Shijiazhuang, Hebei, 050016 P. R. China. Email:
kzhao@wlu.ca

\end{document}